\RequirePackage{fixltx2e}
\documentclass[oneside,english]{amsart}
\usepackage[T1]{fontenc}
\usepackage[latin9]{inputenc}
\usepackage{babel}
\usepackage{textcomp}
\usepackage{amstext}
\usepackage{amsthm}
\usepackage{amssymb}
\usepackage[all]{xy}
\usepackage[unicode=true,
 bookmarks=true,bookmarksnumbered=false,bookmarksopen=false,
 breaklinks=false,pdfborder={0 0 1},backref=false,colorlinks=false]
 {hyperref}
\hypersetup{pdftitle={Boundedness properties of automorphism groups of forms of flag varieties},
 pdfauthor={Attila Guld},
 pdfsubject={Algebraic Geometry}}

\makeatletter
\numberwithin{equation}{section}
\numberwithin{figure}{section}
\theoremstyle{plain}
\newtheorem{thm}{\protect\theoremname}[section]
  \theoremstyle{plain}
  \newtheorem{prop}[thm]{\protect\propositionname}
  \theoremstyle{plain}
  \newtheorem{lem}[thm]{\protect\lemmaname}
  \theoremstyle{plain}
  
  \theoremstyle{remark}
  \newtheorem{rem}[thm]{\protect\remarkname}
  \theoremstyle{definition}
  
  \theoremstyle{plain}
  
  \theoremstyle{definition}
  \newtheorem{defn}[thm]{\protect Definition}
\usepackage{fancyhdr}
\usepackage{ifpdf} 
\ifpdf 
 \IfFileExists{lmodern.sty}{\usepackage{lmodern}}{}
\fi

\makeatother

  \providecommand{\corollaryname}{Corollary}
  \providecommand{\examplename}{Example}
  \providecommand{\lemmaname}{Lemma}
  \providecommand{\propositionname}{Proposition}
  \providecommand{\questionname}{Question}
  \providecommand{\remarkname}{Remark}
  \providecommand{\theoremname}{Theorem}

\DeclareMathOperator{\Fl}{Fl}
\DeclareMathOperator{\Hom}{Hom}
\DeclareMathOperator{\Spec}{Spec}
\DeclareMathOperator{\Aut}{Aut}
\DeclareMathOperator{\GL}{GL}
\DeclareMathOperator{\PGL}{PGL}
\DeclareMathOperator{\Ker}{Ker}
\DeclareMathOperator{\Gal}{Gal}

\DeclareMathOperator{\Bir}{Bir}
\DeclareMathOperator{\m}{min}

\DeclareMathOperator{\ind}{ind}
\DeclareMathOperator{\Lin}{Lin}
\begin{document}

\title[Boundedness of automorphism groups of forms of flag varieties]{Boundedness properties of automorphism groups of forms of flag varieties}

\author{Attila Guld}
\email{guld.attila@renyi.mta.hu}

\thanks{The research was partly supported by the National Research, Development and Innovation Office (NKFIH) Grant No. K120697. 
The project leading to this application has received funding from the European Research Council (ERC) under the European Union's Horizon 2020 research and innovation programme (grant agreement No 741420).}

\address{
R\'enyi Alfr\'ed Matematikai Kutat\'oint\'ezet\\
Re\'altanoda utca 13-15.\\
Budapest, H1053\\
Hungary}

\begin{abstract}
We call a flag variety admissible if its automorphism group is the projective general linear group. (This holds in most cases.)\\
Let $K$ be a field of characteristic $0$, containing all roots of unity.
Let the $K$-variety $X$ be a form of an admissible flag variety. We prove that $X$ is either ruled, or the automorphism group of $X$ is bounded, 
meaning that, there exists a constant $C\in\mathbb{N}$ such that if $G$ is a finite subgroup of $\Aut_K(X)$, then the cardinality of $G$ is smaller than $C$.
\end{abstract}

\maketitle

\section{Introduction}
Before stating our main theorem we need to introduce a couple of definitions and notations.
\begin{defn}[Definition 2.9 in \cite{Po11}]
A group $G$ is called bounded if there exists a constant $C\in\mathbb{N}$ such that every finite subgroup of $G$ has smaller cardinality than $C$. 
\end{defn}

Let $V$ be a finite dimensional vector space (over an arbitrary field). A flag is a strictly increasing sequence of linear subspaces of $V$  (with respect to the the containment order).
By $\Fl (d_1<d_2<...<d_r, V)$ or simply by $\Fl (\mathbf{d},V)$  we denote the flag variety of the sequence of  linear subspaces  of $ V$ (flags) of dimensions
determined by the strictly increasing sequence of nonnegative integers $\mathbf{d}=(d_1,d_2,...,d_r)$, where $d_r\leqq \dim V$. 
We also use the notation $\Fl(\mathbf{d}<\mathbf{e}, V)$  governed by similar logic, using the strictly increasing sequence of nonnegative integers
$\mathbf{d}<\mathbf{e}=(d_1,...,d_p,e_1,...,e_q)$ ($e_q\leqq\dim V$).  If $d_1\geqq n$  then the notation $\mathbf{d}-n$ stands for the strictly increasing sequence of nonnegative integers $\mathbf{d}-n=(d_1-n,...,d_r-n)$.\\
If no confusion can arise we omit the specification of the vector space or the strictly increasing sequence of nonnegative integers or both of them. 
When we say $\Fl(\mathbf{d},V)$ is a flag variety, we implicitly assume that $V$ is a vector space over some field and $\mathbf{d}=(d_1,...,d_r)$ is a strictly increasing sequence of nonnegative integers, where $d_r\leqq\dim V$.\\

\begin{defn}
\label{admissible flag}
We call a flag variety admissible, if its automorphism group is  the projective general linear group, otherwise we call it non-admissible.
\end{defn}

Later on we will see that a flag variety is admissible unless it is  isomorphic to a flag variety $\Fl(d_1<...<d_r, V)$, where $0<d_1$, $d_r<\dim V$, $\dim V\geqq 3$ and $\forall i=1,...,r$ $\:$ $d_i+d_{r+1-i}=\dim V$.
Notice that the conditions $0<d_1$ and $d_r<\dim V$ are technical assumptions, they do not exclude any isomorphism class of flag varieties.
The automorphism group of the non-admissible flag variety $\Fl (\mathbf{d},V)$  is $\PGL(V)\rtimes\mathbb{Z}/2\mathbb{Z}$. (See Theorem \ref{Aut(F)} for further details.)

\begin{defn}
Let $K$ be a field. The $K$-variety $X$ is a form of a flag variety if $X\times\Spec \overline{K}\cong \Fl(\mathbf{d},V_{\overline{K}})$ where $\overline{K}$ is the algebraic closure of $K$, $V_{\overline{K}}$ 
is a finite dimensional $\overline{K}$-vector space, and $\mathbf{d}$ is a strictly increasing sequence of nonnegative integers.
\end{defn}

Now we are ready to state the main theorem of the article.

\begin{thm}
\label{MainThm}
Let $K$ be a field of characteristic $0$, containing all roots of unity. Let the  $K$-variety $X$ be a form of an admissible flag variety. 
Then either the automorphism group $\Aut_K(X)$ is bounded, or $X$ is birational to a direct product variety $Y\times \mathbb{P}^1$, in other words $X$ is ruled.
\end{thm}

Before moving further we introduce another definition from group theory.

\begin{defn}[Definition 2.1 in \cite{Po11}]
\label{Jor}
A group $G$ is called Jordan if there exists a constant $J\in\mathbb{N}$ such that for every finite subgroup $H\leqq G$ there exists an Abelian normal subgroup $A\leqq H$  such that $\vert{H\colon A}\vert<J$.
\end{defn}

In \cite{BZ15b} T. Bandman and Yu. G. Zarhin answered a question of Yu. Prokhorov and C. Shramov (\cite{PS2}) by showing that the birational automorphism group of a conic bundle over a non-uniruled base is Jordan 
when it is not birational to the trivial $\mathbb{P}^1$-bundle over the non-uniruled base. 
One of the major steps in their proof was to show that the birational (and hence the biregular) automorphism group of a non-trivial Brauer-Severi curve is bounded. This follows from our theorem as a special case. 
(They also showed that the cardinalities of the finite subgroups of the automorphism group are bounded by four.)\\
The result on the boundedness of the automorphism groups of non-trivial Brauer-Severi curves was also used by Yu. Prokhorov and C. Shramov 
when they classified three dimensional varieties with non-Jordan birational automorphism groups (\cite{PS16b}).\\

Another aspect of our motivation is that, we would like to investigate conditions which imply that the birational automorphism group of a rationally connected variety is bounded. 
We hope that by regularizing actions of finite subgroups of the birational automorphism groups  (\cite{PS2}, Lemma 3.1), and by the help of the Minimal Model Program, 
this question can be reduced to studying finite subgroups of the automorphism groups of Fano varieties over function fields. 
As a special case we investigated boundedness properties of automorphism groups of forms of flag varieties.\\

The definitions of the Jordan and the boundedness properties was introduced by V. L. Popov. They are closely related. Boundedness implies the Jordan property, 
while a typical strategy for proving that a group is Jordan is to show that the group sits in an exact sequence where the normal subgroup is Jordan and the quotient group is bounded (\cite{Po11}, Lemma 2.11).
A survey of results concerning these properties of groups and the relations between them can be found in \cite{Po} and in Section 2 of \cite{PS2}.\\
Research about investigating Jordan properties for birational and biregular automorphism groups of varieties was initiated by J.-P. Serre in \cite{Se09} and V. L. Popov in \cite{Po11}. 
 Recently many authors have contributed to the subject 
(\cite{BZ15a}, \cite{BZ15b}, \cite{Hu18}, \cite{MZ15}, \cite{Po11}, \cite{Po}, \cite{PS2}, \cite{PS16}, \cite{PS16b}, \cite{Se09}, \cite{Za15}).\\

The idea of our proof is the following. A form of a flag variety can be viewed as a flag variety equipped with a twisted Galois action.  
The automorphism group of the form embeds into the automorphism group of the flag variety, and its action commutes with the twisted Galois action. 
If the automorphism group of the form is not bounded, then the commutation imposes condition on the twisted Galois action. Using this, we may construct a Galois equivariant rational map from the flag variety to a smaller dimensional variety. 
It turns out that this rational map induces a vector bundle structure on the open subset of the flag variety where the map is defined and the twisted Galois action respects the vector bundle structure. 
By results of Galois descent, we descend the vector bundle structure to an open subvariety of the form. This proves our theorem.\\

We use the admissibility hypothesis to construct the \textit{Galois equivarant} rational map from our flag variety to a smaller dimensional variety. Although the rational map can be constructed anyway, 
we use the admissibility condition when we endow the target space with a Galois action which makes the rational map equivaraint.
For a more detailed discussion see Remark \ref{NaturalAction2} and Remark \ref{NANGE}.\\

In general, it is a very hard question to decide whether a variety is ruled or not. Amongst forms of (admissible) flag varieties we can find examples to both cases.\\
Indeed, flag varieties are rational, therefore they are ruled. On the other hand non-trivial Brauer-Severi curves and surfaces provide examples of non-ruled forms of admissible flag varieties.
Non-trivial Brauer-Severi curves are non-ruled  essentially as a consequence of their definition, while the case of non-trivial Brauer-Severi surfaces will be explored in Section \ref{BS2}. 
Here we only state the corresponding theorem.
\begin{thm}
\label{BS2T}
Let $K$ be a field of arbitrary characteristic. Let $X$ be a Brauer-Severi surface over $K$. $X$ is ruled if and only if it is trivial.\\
\end{thm}

The paper is organized in the following way. In Section \ref{Pre} we recall the necessary knowledge about automorphism groups of flag varieties and Galois descent. 
In Section \ref{rational maps} we construct the rational maps which will give us the vector bundle structure. 
It is followed by Section\ref{ComAct}, where we analyze the effect of the commuting group actions when the automorphism group of the form of the flag variety is not bounded. 
Finally, Section\ref{Proof} contains the proof of our theorem. We enclose our article with a discussion on Brauer-Severi surfaces in Section \ref{BS2}.\\

\subsection{Conventions}
Throughout the article we use the following conventions.\\
Unless explicitly stated otherwise all fields are assumed to be of characteristic 0. 
For a field $K$ we use $\overline{K}$ to denote its (fixed) algebraic closure.\\
By a vector space we mean a finite dimensional vector space.
Sometimes in the notation of a vector space we make explicit the field over which the vector space is defined. When we say $V_K$ is a vector space, we mean that $V_K$ is a vector space defined over the field $K$.\\
Let $V$ be a vector space over a field $K$. By $\Lin_K(V)$ we denote the $K$-linear automorphism group of $V$.  (During the article we will encounter situations, where $V$ is a vector space over a field $L$,
 where $K\leqq L$, however we need to consider its $K$-linear automorphism group.)\\
By a variety we mean a separated, integral scheme of finite type over a field.\\ 
Let $X$ be an arbitrary scheme over a field $K$. By $\Aut_K(X)$ we denote the $K$-scheme automorphism group of $X$. (During the article we will encounter situations, where $X$ is a variety over a field $L$, where $K\leqq L$,
and we need to consider its $K$-scheme automorphism group.)\\ 
Let $X$ be a variety, by $\Bir(X)$ we denote the birational automorphism group of $X$.

\subsection{Acknowledgements}
The author is very grateful to E. Szab\'o for all the helpful discussions. The author thanks the referees for their valuable comments. 

\section{Preliminaries}
\label{Pre}

\subsection{Automorphism group of flag varieties}
In this subsection we collect results about automorphism groups of flag varieties.
First, we recall the definition of the automorphism group scheme.  

\begin{defn}
\label{AutSch}
Let $X$ be a scheme over a base scheme $S$. Consider the assignment 
$T\mapsto\Aut_T(X\times T)$ between $S$-schemes and abstract groups. It gives rise to a contravariant functor $A_X:(Sch/S)^{op}\to\underline{Gr}$ from the category of $S$-schemes to the category of groups. 
($(Sch/S)^{op}$ denotes the opposite category of the category of $S$-schemes). If $A_X$ can be represented by an $S$-scheme $Y$, then we call $Y$ the automorphism group scheme of $X$, and denote it by $Aut_S(X)$ or simply by $Aut(X)$. 
(In case of $S=\Spec K$, for some field $K$, we also use the notation $Aut_K(X)$.)
\end{defn}

\begin{rem}
\label{PofAut}
Note that the definition implies that $Aut_T(X\times T)\cong Aut_S(X)\times T$ for any $S$-scheme $T$ (by the adjoint property of restriction and extension of scalars).\\
It is also worth pointing out that an immediate consequence of the definition is the following. For a $K$-scheme $X$, if $Aut(X)$ exits, then the group of its $K$-rational points is isomorphic to the automorphism group of $X$,
 in formula $(Aut_K(X))(K)\cong\Aut_K(X)$.
\end{rem}

The following theorem of H. Matsumura and F. Oort secures the existence of the automorphism group schemes for flag varieties (Theorem 3.7 in \cite{MO}).

\begin{thm}
Let $K$ be a field of arbitrary characteristic, and let $X$ be a proper $K$-scheme. The automorphism group scheme $Aut(X)$ exists and it is of locally finite type over $K$.
\end{thm}

Armed with the concept of automorphism group schemes, we can make our first step towards describing the automorphism groups of flag varieties.

\begin{prop}
\label{PGL}
Let $K$ be a field, $V$ be a $K$-vector space and $\Fl(\mathbf{d},V)$ be a $K$-flag variety. The group scheme $PGL(V)$ is a closed subscheme of $Aut_K(\Fl(\mathbf{d},V))$.
\end{prop}

\begin{proof}
Clearly the functor of points of the group scheme of the projective general linear group $\Hom(-,PGL(V))$ is a subfunctor of $A_{\Fl(\mathbf{d},V)}$ defined in Definition \ref{AutSch}. 
Therefore we have a morphism of group schemes $\varphi:PGL(V)\to Aut_K(\Fl(\mathbf{d},V))$.\\
The kernel of $\varphi$ is trivial. Indeed, $\PGL(V\otimes L)$ embeds into $\Aut_L(\Fl(\mathbf{d},V)\times \Spec L)\cong\Aut_L(\Fl(\mathbf{d},V\otimes L))$ for any field extension $L\vert K$. 
Therefore the kernel has a unique rational point over any field. Since we work in characteristic $0$, this implies that the kernel is trivial (by smoothness).\\
Since the kernel is trivial and $\varphi$ is a smooth morphism (as the characteristic is $0$), $\varphi$ is a closed immersion (Lemma 38.7.8, \cite{Stack}).
\end{proof}

\begin{rem}
\label{na}
Consider flag varieties of the form $\Fl(d_1<...<d_r, V)$, where $0<d_1$, $d_r<\dim V$, $\dim V\geqq 3$ and $\forall i=1,...,r$  $d_i+d_{r+1-i}=\dim V$.
(Notice that the conditions $0<d_1$ and $d_r<\dim V$ are technical assumptions, they do not exclude any isomorphism class of flag varieties.)\\
In the next theorem we will show that non-admissible flag varieties are exactly flag varieties of the above form. In this remark we will construct an order two automorphism  for them, called $\tau$,  which lies outside $\PGL(V)$ and normalizes it.
This strengthens the previous proposition, since the existence of $\tau$ implies that in case of flag varieties of the above form $PGL(V)\rtimes \mathbb{Z}/2\mathbb{Z}$ is a closed subscheme of the automorphism group scheme.\\
The involution $\tau$ can be constructed in the following way.
For an arbitrary flag variety $\Fl(\mathbf{e}, W)$ (not necessarily of the form considered in the beginning of the remark) we can examine the dual map:
\begin{gather*}
*: \Fl (e_1<e_2<...<e_q, W)\to \Fl (m-e_q<m-e_{q-1}<...<m-e_1, W^*)\\
U_1<U_2<...<U_q\mapsto {U_q}^{\perp}<{U_{q-1}}^{\perp}<...<{U_1}^{\perp},
\end{gather*}
where $m=\dim W$, $W^*$ is the dual space of $W$ and for an arbitrary linear subspace $U\leqq W\;$ $U^\perp=\{\varphi\in W^*|\varphi|_U\equiv0\}$ is the annihilator subspace.\\
Consider a flag variety $\Fl(\mathbf{d}, V)$ of the form introduced in the beginning of the remark, and fix a linear automorphism $j_0:V^*\to V$ such that $j_0^{-1}$ maps a (fixed)  basis of $V$ to its dual basis ($V^*$ denotes the dual space of $V$). 
$j_0$ induces an isomorphism $j:\Fl(\mathbf{d}, V^*)\to\Fl(\mathbf{d}, V)$. With a little amount of work it can be checked that the automorphism $\tau=j\circ *$ is an involution outside the projective general linear group, 
and that $\tau$ normalizes the projective general linear group. 
(If $\dim V=2$, then $\tau$ would be an element of the projective general linear group.)
\end{rem}

Our next tool  is the result of H. Tango (Theorem 2 in\cite{Ta}). By the use of Schubert calculus he gave a description of the automorphism groups of flag varieties over algebraically closed fields (of arbitrary characteristic).\\
Just as in the previous remark, when we state the next theorem we will use the technical assumption that a flag does not contain the trivial linear subspace and the whole vector space (i.e. $0<d_1$ and $d_r<\dim V$).
\begin{thm}
\label{Aut(F)}
Let $K$ be a field, $V$ be a $K$-vector space and let $\mathbf{d}$ denote a strictly increasing sequence of integers $d_1 <...<d_r$, where $0<d_1$ and $d_r<\dim V$.
The automorphism group of the $K$-flag variety  $\Fl(\mathbf{d},V)$ is $\PGL(V)$ (with its natural action on the variety), except the case when
$3\leqq \dim V$ and $d_i+d_{r-i+1}=\dim V$ for all $i=1,...,r$. In this later case the automorphism group is $\PGL(V)\rtimes\mathbb{Z}/2\mathbb{Z}$.
\end{thm}

\begin{proof}
When $K$ is algebraically closed, this is Tango's theorem  (Theorem 2 in\cite{Ta}).\\
As a first step towards describing the case when $K$ is not algebraically closed, 
we prove a stronger version of the theorem. We prove that the automorphism group scheme has the form which naturally corresponds to the form of the automorphism group described by the theorem, 
i.e. it is $PGL(V)$ or a $PGL(V)\rtimes\mathbb{Z}/2\mathbb{Z}$ accordingly.\\
Observe that the automorphism group scheme  of a complex flag variety has the desired form. Indeed, by Tango's theorem, 
the automorphism group of a complex flag variety is either $\PGL(V)$ or $\PGL(V)\rtimes\mathbb{Z}/2\mathbb{Z}$.
Therefore the group of the closed points of the automorphism group scheme of a complex flag variety gives back the groups described by our theorem. 
Combining this fact with the result of Proposition \ref{PGL}, which states that $PGL(V)$ is a closed subscheme of the automorphism group scheme of a complex flag variety,
we can conclude our claim for the automorphism group scheme of an arbitrary $\mathbb{C}$-flag variety.\\
As the next step, note that it is enough to show that the our claim for automorphism group schemes holds for flag varieties over $\mathbb{Q}$. Indeed, let $\Fl(\mathbf{d},V_K)$ be an arbitrary flag variety over an arbitrary field $K$. 
By choosing a basis of $V_K$, we can find a $\mathbb{Q}$-vector space $W_{\mathbb{Q}}$ such that $W_\mathbb{Q}\otimes K\cong V_K$ and $\Fl(\mathbf{d},V_K)\cong\Fl(\mathbf{d},W_\mathbb{Q})\times\Spec K$. 
Hence Remark \ref {PofAut} implies that $Aut_{K}(\Fl(\mathbf{d},V_K))\cong Aut_{\mathbb{Q}}(\Fl(\mathbf{d},W_\mathbb{Q}))\times\Spec K$. The result follows, as $PGL(W_\mathbb{Q})\times\Spec K\cong PGL(V_K)$.\\
Let $\Fl(\mathbf{d},U_{\mathbb{Q}})$ be an arbitrary flag variety over $\mathbb{Q}$. Since base changing the ground field does not affect dimensions, 
the group schemes $PGL(U_\mathbb{Q})$ and $Aut_\mathbb{Q}(\Fl(\mathbf{d},U_\mathbb{Q}))$ has the same dimension by the complex case. 
As $PGL(U_\mathbb{Q})$ is connected, and it is a closed group scheme of $Aut_\mathbb{Q}(\Fl(\mathbf{d},U_\mathbb{Q}))$ (Proposition \ref{PGL}), we conclude that it is the identity component.\\
A similar logic applies to the number of connected components. Indeed, the number of connected components cannot decrease after base changing the ground field. 
Hence using the case of complex flag varieties and the result of Remark \ref{na}, our claim for the automorphism group schemes follows.\\
Now we can turn back our attention to the automorphism groups. We conclude our proof by taking rational points of the automorphism group schemes and using Remark \ref{PofAut}. 
\end{proof}

\begin{rem}
\label{CharAdmFl} 
Notice that Theorem \ref{Aut(F)} gives a new characterization of admissible flag varieties. This new characterization only uses dimensions of the linear subspaces of a flag of the variety and the dimension of the under lying vector space. 
\end{rem}

\begin{rem}
\label{NaturalAction}
The projective general linear group has a natural action on the set of  $d$-dimensional linear subspaces of the underlying vector space (for every fixed $d$, where $0\leqq d\leqq n$ and $n$ is the dimension of the underlying vector space). 
These actions are compatible with the action of the projective general linear group on the flags of the vector space. Sometimes we use this observation without further notice.
A similar statement holds for twisted Galois actions on admissible flag varieties (check Remark \ref{NaturalAction2}).\\
The automorphism group $\PGL(V)\rtimes \mathbb{Z}/2\mathbb{Z}$ of a non-admissible flag variety also has a natural action on the set of the union of $d$-dimensional and $(n-d)$-dimensional linear subspaces of the underlying vector space.
However some group elements swap the dimensions. A similar kind of claim can be formulated for the twisted Galois actions on non-admissible flag varieties (check Remark \ref{NaturalAction2}).
\end{rem}

\subsection{Galois descent}
We collect results about Galois descent and fields in general. First we start with a couple of technical claims. The next lemma can be proved by standard techniques using the finiteness condition built in the definition of a variety.
\begin{lem}
\label{RedFin}
Let $K$ be a field. Let $X$ and $Y$ be $\overline{K}$-varieties and $\varphi:X\to Y$ be a morphism between them. There exists a finite Galois extension $L|K$ such that $X$, $Y$ and $\varphi$ are defined over $L$.
More precisely, there exists $X'$, $Y'$ $L$-varieties and $\varphi':X'\to Y'$ morphism between them such that $X'\times \Spec \overline{K}\cong X$, $ Y'\times \Spec \overline{K}\cong Y$ and $\varphi'\times id\cong \varphi$.
\end{lem}

\begin{rem}
\label{finiteness}
Let the $K$-variety $X$ be a form of  a flag variety. By Lemma \ref{RedFin}, we can find a finite Galois extension $L|K$ such that $X\times\Spec L\cong\Fl(\mathbf{d},W_L)$.
Indeed, applying the lemma to the isomorphism between $X\times\Spec \overline{K}$ and $\Fl(\mathbf{d},V_{\overline{K}})$ proves the claim.
\end{rem}

\begin{rem}
If the $K$-variety $X$ is a form of a flag variety, then $X$ is projective. Indeed if $X\times\Spec\overline{K}$ is projective, then the same holds for $X$ as well (Proposition 14.55  in \cite {GW}).
\end{rem}

\begin{defn}
\label{SF}
Let $K$ be  a field, and let the $K$-variety $X$ be a form of a flag variety. If $L\vert K$ is a field extension such that $X\times\Spec{L}\cong \Fl(\mathbf{d},V_L)$, then we call $L$ a splitting field for $X$.
By Remark \ref{finiteness} $L|K$ can chosen to be a finite Galois extension.
\end{defn}

Now we turn our attention to results about descents. This part of the subsection is mainly based on \cite{Ja}.
\begin{defn}
Let $L\vert K$ be a Galois extension with Galois group $\Gamma$. We call a pair $(X,T)$ a quasi-projective $L$-scheme equipped with a twisted Galois action, if $X$ is a quasi-projective $L$-scheme and $T:\Gamma\to \Aut_K(X)$ is a group homomorphism satisfying the following commutative diagram (for every $\sigma\in\Gamma$):
 \[
\xymatrix{
X \ar[d] \ar[r]^{T(\sigma)} & X \ar[d]\\
\Spec L\ar[r]^{S(\sigma)} & \Spec L
}
\]
where $S(\sigma):\Spec L\to \Spec L$ is the morphism of schemes induced by $\sigma^{-1}:L\to L$.\\
If no confusion can arise we denote the pair $(X,T)$ simply by $X$.
(Observe that $S(\sigma)$ is induced by $\sigma^{-1}$ since there is an \textit{antiequivalence} of categories between affine schemes and rings. Therefore using the inverse is necessary to define an action of the Galois group.)  
\end{defn}

The following theorem can be found in \cite{Ja} (Theorem 2.2.b).
\begin{thm}
\label{equivalence}
Let $L\vert K$ be a finite Galois extension with Galois group $\Gamma$. There is an equivalence between the category of quasi-projective $K$-schemes and the category of quasi-projective $L$-schemes equipped with a twisted $\Gamma$-action.
The equivalence functor is given by $X\mapsto X\times \Spec L$.
\end{thm}

\begin{rem}
Since the theorem is about equivalence of categories, it also says that Galois equivariant morphisms descends to morphism of the underlying $K$-schemes.
\end{rem}

\begin{defn}
\label{Galois action}
Let $L\vert K$ be a Galois extension with Galois group $\Gamma$ and let $V_L$ be an n-dimensional vector space over $L$. Let $b=(v_1,...,v_n)$ be a basis of $V_L$.
There is a twisted Galois action $A_b:\Gamma\to \Lin_K(V_L)$ defined by
\begin{gather*}
A_b(\sigma):V_L\to V_L \\
v=\alpha_1 v_1+\alpha_2 v_2+...+\alpha_n v_n\mapsto \sigma(\alpha_1) v_1+\sigma(\alpha_2) v_2+...+\sigma(\alpha_n) v_n, 
\end{gather*}
where the $\alpha_i$'s are coefficients from the field $L$ ($i=1,...,n$) and $\sigma\in\Gamma$ is an arbitrary element of the Galois group. 
For a flag variety $\Fl(\mathbf{d},V_L)$ this induces a twisted Galois action, denoted by $B_b:\Gamma\to\Aut_K(\Fl(\mathbf{d},V_L))$.
 \[
\xymatrix{
\Fl(\mathbf{d},V_L) \ar[d] \ar[r]^{B_b(\sigma)} & \Fl(\mathbf{d},V_L) \ar[d]\\
\Spec L\ar[r]^{S(\sigma)} & \Spec L
}
\]
It might seem counterintuitive that the diagram contains $S(\sigma)$, which is induced by $\sigma^{-1}$. However after realizing that this means we pull back functions using $\sigma^{-1}$ , we can also realize that it forces us to use $\sigma$ when 
we want to `push forward' scalars.\\
Notice that if $T$ is an arbitrary twisted Galois action on $\Fl(\mathbf{d},V_L)$, then for every $\sigma\in\Gamma$  the morphism $T(\sigma){B_b(\sigma)}^{-1}$ is an element of the automorphism group of the flag variety,
therefore $T(\sigma)$ can be written as $T(\sigma)=a_\sigma\circ B_b(\sigma)$ where $a_\sigma\in \Aut_L(\Fl(\mathbf{d},V_L))$. (Of course $a_\sigma$ also depends on the basis $b$, although we decided to omit it in the notation.)
\end{defn}

\begin{rem}
\label{NaturalAction2}
Let $L\vert K$ be a Galois extension with Galois group $\Gamma$, $V_L$ be an $L$-vector space and $\Fl(\mathbf{d}, V_L)$ be an $L$-flag variety with a twisted Galois action $T:\Gamma\to\Aut_K(\Fl(\mathbf{d},V_L))$. 
Choose a basis of $V_L$, denote it by $b$. We saw in the previous definition that $T(\sigma)=a_\sigma\circ B_b(\sigma)$.\\ 
Assume that the flag variety is admissible, then $a_\sigma$ is an element of $\PGL(V_L)$, therefore it has a natural action on the set of  linear subspaces of $V_L$. 
$B_b(\sigma)$ can also be endowed with a natural action on the set of  linear subspaces of $V_L$ (via $A_b(\sigma)$). This enables us to endow $T(\sigma)$ with a natural action on the set of linear subspaces of $V_L$ . 
Moreover, this action is compatible with the action of $T(\sigma)$ on the flags. In formula
$T(\sigma)(Z_1<..<Z_r)=(c_\sigma\circ A_b(\sigma))(Z_1)<...<(c_\sigma\circ A_b(\sigma))(Z_r)$ for any flag $Z_1<...<Z_r\in \Fl(\mathbf{d}, V_L)$. Sometimes we use this observation without further notice.\\
If the flag variety is non-admissible then we can also formulate a similar claim. 
However if $a_\sigma\not\in \PGL(V)$ then the $d_i$-dimensional linear subspace of the image flag $T(\sigma)(Z_1<..<Z_r)$ depends on the  $(\dim V-d_i)$-dimensional linear subspace of the flag $Z_1<...<Z_r$.
The existence of these `dimension-swapping' morphisms can pose problems when we try to construct Galois equivariant morphisms from non-admissible flag varieties.
 \end{rem}

\begin{rem}
\label{H1}
 If $L\vert K$ is a finite Galois extension (such that $L\leqq\overline{K}$) with Galois group $\Gamma$ and $V_L$ is an $L$-vector space, then $\{\sigma\mapsto a_{\sigma}\}$ gives an element in the first group cohomology
$\mathrm{H}^1(\Gamma,\Aut_L(\Fl(\mathbf{d},V_L))$. The elements of the first group cohomology are in 1-to-1 correspondence with the forms of $\Fl(\mathbf{d},V_L)\times\Spec\overline{K}\cong\Fl(\mathbf{d},V_L\otimes\overline{K})$ split by $L$. For further informations on this, see Theorem 14.88 in\cite{GW}. Also Theorem 3.6 and Theorem 4.5 in \cite{Ja} give results of similar flavour in the case of Brauer-Severi varieties.
\end{rem}

\begin{thm}
\label{G-vb}
Let $L\vert K$ be a finite Galois extension with Galois group $\Gamma$. Let $X,Y$ be quasi-projective $L$-schemes equipped with twisted Galois actions, 
and let $\phi:X\to Y$ be a Galois equivariant morphism of $L$-schemes such that the triple $(X,Y,\phi)$ forms a vector bundle. 
Moreover, let the Galois action respect the vector bundle structure (respect the addition and twist the multiplication by scalar operations).  
Then there exist $X',Y'$ quasi-projective $K$-schemes and $\phi':X'\to Y'$ morphism of $K$-schemes such that $(X',Y',\phi')$ forms a vector bundle and $X'\times \Spec L\cong X$, $ Y'\times \Spec L\cong Y$, $\phi'\times id\cong \phi$.
\end{thm}

\begin{proof}
By Theorem 2.2.c in \cite{Ja}, a locally free sheaf of finite rank $\mathcal{E}$ equipped with a Galois action compatible with the Galois action on the underlying quasi-projective $L$-scheme $Y$ comes from
 a locally free sheaf (of the same rank) on the quasi-projective $K$-scheme $Y'$, where $Y'\times \Spec L\cong Y$. Since there is a 1-to-1 canonical correspondence between finite rank vector bundles and locally free sheaves of finite rank, 
 the result follows. 
\end{proof}

\section{Rational maps of flag varieties}
\label{rational maps}
Let $K$ be a field and $\overline{K}$ be its algebraic closure. Let $V$ be a vector space over $\overline{K}$. Assume $V=V_1\oplus V_2$  is a direct sum decomposition, $\dim V=n$ and $\dim V_i=n_i$ ($i=1,2$). 
Consider the strictly increasing sequences of nonnegative integers $\mathbf{d}=(d_1,d_2,...,d_p)$ and $\mathbf{e}=(e_1,e_2,...,e_q)$, where $d_p\leqq n_1<e_1$  and $e_q\leqq n$.  We are going to investigate the rational maps 
\begin{gather*}
\phi_1:\Fl(\mathbf{d},V)\dashrightarrow \Fl(\mathbf{d},V_1)\\
Z_1<...<Z_p \mapsto pr(Z_1)<...<pr(Z_p)
\end{gather*} 
where $pr:V\to V_1$  is the projection along $V_2$, $Z_i$'s ($i=1,...,p$) are the vector spaces forming the flag ($\dim Z_i=d_i$),
 \begin{gather*}
\phi_2:\Fl(\mathbf{e},V)\dashrightarrow \Fl(\mathbf{e}-n_1,V_2)\\  
W_1<...<W_q \mapsto W_1\cap V_2<...<W_q\cap V_2
\end{gather*}
where 
$W_j$'s ($j=1,...,q$) are the vector spaces forming the flag ($\dim W_j=e_j$),
\begin{gather*}
\psi:\Fl(\mathbf{d}<\mathbf{e},V)\dashrightarrow \Fl(\mathbf{d},V_1)\times \Fl(\mathbf{e}-n_1,V_2)\\
Z_1<...<Z_p<W_1<...<W_q \mapsto (pr(Z_1)<...<pr(Z_p), W_1\cap V_2<...<W_q\cap V_2)
\end{gather*}
where $Z_i$'s ($i=1,...,p$) and $W_j$'s ($j=1,...,q$) are the vector spaces forming the flag ($\dim Z_i=d_i$, $\dim W_j=e_j$).\\
 Clearly all of these are rational maps. 
$\phi_1$ is defined on the open subvariety
 \[U_1=\{Z_1<...<Z_p\in\Fl(\mathbf{d},V)\vert Z_p\cap V_2=\{0\} \},\]
$\phi_2$ is defined on the open subvariety 
\begin{multline*}
U_2=\{W_1<...<W_q\in\Fl(\mathbf{e},V)\vert W_1\pitchfork V_2\}=\\
\{W_1<...<W_q\in\Fl(\mathbf{e,V})\vert W_1+V_2=V\},
\end{multline*} 
and $\psi$ is defined on the open subvariety 
\begin{multline*}
U=\{Z_1<...<Z_p<W_1<...<W_q\in\Fl(\mathbf{d}<\mathbf{e},V)\vert\\
Z_p\cap V_2=\{0\}, W_1+V_2=V\}.
\end{multline*}
We can check that $U_1$, $U_2$ and $U$ are open subvarieties.
Indeed, let $\alpha_p$ be the tautological vector bundle on the flag variety $\Fl(\mathbf{d},V)$ corresponding to the $d_p$-dimensional linear subspaces of $V$, 
and let $\beta_1$ be the tautological vector bundle on the flag variety $\Fl(\mathbf{e}, V)$ corresponding to the $e_1$-dimensional linear subspaces of $V$. 
Let $\rho$ be the global section of the hom-vector bundle  $Hom_{\overline{K}}(\alpha_p, V/V_2)$ induced by the projection $V\to V/V_2$, and 
let $\tau$ be the global section of the hom-vector bundle  $Hom_{\overline{K}}(\beta_1, V/V_2)$ induced by the projection $V\to V/V_2$.
$U_1$ is the open locus where $\rho$ has maximal rank, while
$U_2$ is the open locus where $\tau$ has maximal rank.  
Combining the above arguments, we can also show that $U$ is an open subvariety.

\begin{prop}
\label{vb}
Using the notation introduced in this section, the following holds.
The triples $(U_1,\Fl(\mathbf{d},V_1),\phi_1)$, $(U_2,\Fl(\mathbf{e}-n_1,V_2),\phi_2)$ and $(U,\Fl(\mathbf{d},V_1)\times\Fl(\mathbf{e}-n_1,V_2),\psi)$ form vector bundles.
\end{prop}
\begin{proof}
To see this, first, consider the fiber of $\phi_1$ over an arbitrary flag $S_1<...<S_p\in \Fl(\mathbf{d},V_1) $. Notice that if $Z_1<...<Z_p$ is in the fiber, then it is uniquely determined by $Z_p$. 
Indeed $pr$ induces and isomorphism between $Z_p$ and $S_p$, so there is a unique linear subspace of $Z_p$ which maps to $S_i$ ($i=1,...,p$).\\ 
The $d_p$-dimensional linear subspaces of $V$ which are mapped to $S_p$ are parametrized by $\Hom(S_p,V_2)$. If $f\in\Hom(S_p,V_2)$, then the graph of $f$ considered as a linear subspace of $V$ determines $Z_p$.
More precisely 
\begin{equation} 
\label{S}
Z_p=\{v+f(v)\vert v\in S_p\}.
\end{equation}
On the other hand, $Z_p$ gives an element in $\Hom(S_p, V_2)$ by the composition $pr'\circ t$, where $t:S_p \to Z_p$ is the inverse of the linear isomorphism between $Z_p$ and $S_p$ induced by $pr$
and $pr':V\to V_2$ is the projection along $V_1$. These two constructions are inverse to each other, which shows our claim on the fiber.\\
The argument can be globalized.  It shows that $U_1$ is isomorphic to the total space of the vector bundle corresponding to the locally trivial sheaf of finite rank $Hom_{\mathcal{O}}(\gamma_p,V_2\otimes\mathcal{O})$, 
where $\gamma_p$ is the sheaf of sections of the tautological bundle  of the flag variety $\Fl(\mathbf{d},V_1)$ corresponding to the $d_p$-dimensional linear subspaces and $\mathcal{O}$ is the structure sheaf of $\Fl(\mathbf{d},V_1)$.\\
A similar argument shows that $U_2$ is the total space of the vector bundle corresponding to the locally trivial sheaf of finite rank $Hom_{\mathcal{O}}(V_1\otimes\mathcal{O},(V_2\otimes \mathcal{O})/ \eta_1)$,
where $\eta_1$ is the sheaf of sections of the tautological bundle of the flag variety $\Fl(\mathbf{e}-n_1,V_2)$ corresponding to the $e_1-n_1$-dimensional linear subspaces and $\mathcal{O}$ is the structure sheaf of $\Fl(\mathbf{e}-n_1, V_2)$. 
(By the properties of $\eta_1$, $(V_2\otimes\mathcal{O})/\eta_1$ is a locally trivial sheaf of finite rank).\\
Indeed, again, notice first that an element $W_1<...<W_q\in \Fl(\mathbf{e},V)$, which is in the fiber over $T_1<...<T_q\in \Fl(\mathbf{e}-n_1,V_2)$, is uniquely determined by $W_1$. Since $W_j$ should contain both $W_1$ and $T_j$,
moreover $W_1\cap T_j=W_1\cap V_2=T_1$, we have $W_j=W_1+T_j$ by dimension counting ($j=1,...,q$).\\
The $e_1$-dimensional  linear subspaces $W_1<V$, such that $W_1\cap V_2=T_1$, are parametrized by $\Hom(V_1,V_2/T_1)$. For $g\in \Hom(V_1,V_2/T_1)$ consider the linear subspace 
\begin{equation}
\label{W'}
W'_1=\{u(v)+g(v)\in V/T_1\vert v\in V_1\} 
\end{equation}
of the quotient space $V/T_1$, where we use $u$ to denote the quotient morphism $u:V\to V/T_1$. Finally, let
\begin{equation} 
\label{W}
W_1=u^{-1}(W'_1).
\end{equation} 
Conversely, assume $W_1$ is given. Identify $V_1$ , $V_2/T_1$ and $W_1/T_1$ with linear subspaces of $V/T_1$. Let $p_1:V/T_1\to V_1$ be the projection along $V_2/T_1$, and $p_2:V/T_1\to V_2/T_1$ be the projection along $V_1$.
$p_1$ induces an isomorphism $q_1:W_1/T_1\to V_1$. Let $g\in\Hom(V_1,V_2/T_1)$ be $g=p_2\circ q_1^{-1}$. These two constructions are inverse to each other. The argument globalizes. This proves our claim.\\
For $\psi$ we can use similar constructions. The fiber over $(S_1<...<S_p, T_1<...<T_q)$ is parametrized by a linear subspace $E<\Hom(S_p,V_2)\times \Hom(V_1,V_2/T_1)$ for which the constructions, 
described in the previous paragraphs, yield linear subspaces $Z_p$ and $W_1$  satisfying $Z_p<W_1$.\\ 
This condition is equivalent to $Z_p\leqq W_1$ by dimension counting, which in turn is equivalent to $Z_p+T_1\leqq W_1$. Using the projection $u:V\to V/T_1$, our condition is $u(Z_p)\leqq u(W_1)$. 
By the construction of $Z_p$ and $W_1$ from $(f,g)\in \Hom(S_p,V_2)\times \Hom(V_1,V_2/T_1)$, the condition is equivalent to
 \[(u+u\circ f)(S_p)\leqq (u+g)(V_1).\] 
Consider the identification $V/T_1=V_1\oplus V_2/T_1$. 
\begin{multline*}
\{(v,u\circ f(v))\in V_1\oplus V_2/T_1\vert v\in S_p\}=(u+u\circ f)(S_p)\leqq \\ 
(u+g)(V_1)=\{(v,g(v))\in V_1\oplus V_2/T_1\vert v\in V_1\} 
\end{multline*}
This is equivalent to $u\circ f= g\circ i$, where $i:S_p\to V_1$ is the inclusion map. Let F be the surjective map of linear spaces given by
\begin{gather*}
F:\Hom(S_p,V_2)\times \Hom(V_1,V_2/T_1)\to \Hom(S_p,V_2/T_1)\\
(f,g)\mapsto u\circ f-g\circ i.
\end{gather*}
Then $E=\Ker F$ . Once again, this construction globalizes.
$U\subset \Fl(\mathbf{d}<\mathbf{e}, V)$ is the total space of the vector bundle corresponding to a locally trivial sheaf of finite rank $\mathcal{E}$. 
($\mathcal{E}$ is the kernel of a surjective morphism of locally trivial sheaves of finite rank, hence it is locally trivial of finite rank.)
\end{proof}

\begin{rem}
Let $A_1$, $A_2$ and $A$ be the complements of the open subvarieties $U_1$, $U_2$ and $U$ in the appropriate flag varieties and endow them with the reduced scheme structure.
A short calculation shows that
\begin{align*}
A_1=\{Z_1<...<Z_p\in\Fl(\mathbf{d},V)\vert \dim (Z_p\cap V_2)>0\},
\end{align*}
\begin{align*}
A_2=\{W_1<...<W_q\in\Fl(\mathbf{e},V)\vert \dim(W_1\cap V_2)>e_1-n_1\},
\end{align*}
\begin{multline*}
A=\{Z_1<...<Z_p<W_1<...<W_q\in\Fl(\mathbf{d}<\mathbf{e},V)\vert\\
\dim (Z_p\cap V_2)>0 \; \; \text{or} \; \dim(W_1\cap V_2)>e_1-n_1\}.
\end{multline*}
Hence $A_1$, $A_2$ and $A$ are union of Schubert cells. Recall that an $a$-dimensional Schubert cell is isomorphic to the $a$-dimensional affine space $\mathbb{A}^a$. 
(For more details on Schubert cells the interested reader can consult with Chapter 10.2 in \cite{Fu97}.) 
\end{rem}

\begin{rem}
\label{L-vb}
By Lemma \ref{RedFin} we can find a finite Galois extension $L|K$  (where $L\leqq \overline{K}$) such that $\phi_1, \phi_2$, $\psi$, $U_1$, $U_2$, $U$  and $A_1$, $A_2$, $A$ are defined over $L$. 
Moreover we can require that, the decompositions of $A_1$, $A_2$ and $A$ into the union of Schubert cells exist over  the field $L$. 
In particular, this implies that, the sets of $L$-rational points are dense in  $A_1$, $A_2$ and $A$ (as the same hold for the affine spaces).\\
Furthermore, since a vector bundle structure over a variety can be defined only using finitely many elements from the ground field, we can secure that Proposition \ref{vb} also holds over $L$.\\
During Section \ref{Proof} we will work over a finite Galois extension $L|K$ and use the notation introduced in this section (more precisely its corresponding counterpart which is defined over the field $L$).
\end{rem}

\section{Group actions on forms of flag varieties}
\label{ComAct}
We recall some theorems about birational automorphism groups. To start with, we introduce the notion of strongly Jordan groups. It first appeared  in \cite{BZ17}.

\begin{defn}[Definition 1.1 in\cite{BZ17}]
\label{sJ}
A group $G$ is called strongly Jordan if it is Jordan, and there exists a constant $r\in\mathbb{N}$ such that every finite Abelian subgroup $A\leqq G$ can be generated by $r$ elements, 
in other words the rank of an arbitrary finite Abelian subgroup is smaller than $r$.
\end{defn}

\begin{thm}
\label{strJor}
Let $X$ be a variety. If $X$ is either rationally connected or non-uniruled, then the birational automorphism group $\Bir(X)$ is strongly Jordan.
\end{thm}
\begin{proof}
If $X$ is rationally connected then the birational automorphism group is Jordan by Theorem 1.8 of \cite{PS16} and Theorem 1.1 of \cite{Bi16}. If $X$ is non-uniruled then the birational automorphism group is Jordan by Theorem 1.8 of \cite{PS2}.\\
Furthermore, Remark 6.9 of \cite{PS2} and Theorem 1.1 of \cite{Bi16} shows that the ranks of the finite Abelian subgroups of the birational automorphism group of an arbitrary variety is bounded by a constant depending only on the variety.\\
Putting together these results proves the theorem. 
\end{proof}

\begin{thm}
Let $X$ be a variety. Let $G\leqq\Bir(X)$ be an arbitrary subgroup of the birational automorphism group. Assume that $G$ is not bounded.
Then there exist elements of $G$ of finite and arbitrary large order. 
\end{thm}
\begin{proof}
Assume that there exists a constant $N\in\mathbb{N}$ such that if $g\in G$ is an element of finite order, then the order of $g$ is smaller than $N$. We will show that this implies the boundedness of $G$.\\
By Proposition 6.2 of \cite{PS2} for an arbitrary variety $X$ (using the MRC-fibration) we can fix a rationally connected variety $X_{rc}$ over some function field 
and a non-uniruled variety $X_{nu}$ over the ground field such that an arbitrary finite subgroup $G_0\leqq G(\leqq\Bir(X))$
is an extension of finite groups $G_{rc}$ and $G_{nu}$, where $G_{rc}\leqq\Bir(X_{rc})$ and $G_{nu}\leqq\Bir(X_{nu})$. 
(Note that, if the MRC-fibration is trivial, then either the group $G_{rc}$ or the group $G_{nu}$ is trivial, which does not pose any problem in our argument.)\\
We know that $\Bir(X_{rc})$ and $\Bir(X_{nu})$ are strongly Jordan groups. Denote the corresponding Jordan constants by $J_{rc}$ and $J_{nu}$ respectively, 
and denote the constants bounding the ranks of finite Abelian subgroups by $r_{rc}$ and $r_{nu}$ respectively. Since $X_{rc}$ and $X_{nu}$ only depend on $X$, 
$J_{rc}$, $J_{nu}$ and $r_{rc}$, $r_{nu}$ only depend on $X$ as well.\\
We will use the following easy observation. Let $A$ be a finite Abelian group. Assume that $A$ can be generated by $r$ elements and the order of an arbitrary element $a\in A$ is smaller than $N$. Then the cardinality of $A$ is smaller than $r^N$.\\
$G_{rc}$ is isomorphic to a finite subgroup of $G_0$, therefore the order of an arbitrary element of $G_{rc}$ is smaller than $N$. 
$G_{rc}$ has an Abelian subgroup of rank at most $r_{rc}$  and of index smaller than $J_{rc}$. Hence $|G_{rc}|<J_{rc}r_{rc}^N$.\\
$G_{nu}$  is the homomorphic image of $G_0$, therefore the order of an arbitrary element of $G_{nu}$ is smaller than $N$. 
$G_{nu}$ has an Abelian subgroup of rank at most $r_{nu}$ and of index smaller than $J_{nu}$. Hence $|G_{nu}|<J_{nu}r_{nu}^N$.\\
Therefore $|G_0|<J_{rc}J_{nu}(r_{rc}r_{nu})^N$. Since $G_0$ was an arbitrary finite subgroup of $G$, and all constants depend only on $X$, $G$ is bounded. This contradiction finishes the proof. 
 \end{proof}

\begin{rem}
A similar argument proves the following claim.
Let $X$ be a variety, then there exists a constant $m\in\mathbb{N}$, depending only on $X$, such that any finite subgroup of the birational automorphism group $\Bir(X)$
can be generated by $m$ elements.
\end{rem}

\begin{defn}
\label{NatGalAct}
Let $K$ be a field and let the $K$-variety $X$ be a form of a flag variety. Let $L$ be a splitting field for $X$ such that $L|K$ is a Galois extension. Fix an isomorphism $\varphi: \Fl(\mathbf{d},V)\to X\times\Spec L$. 
Let  $T:\Gal(L|K)\to \Aut_K(\Fl(\mathbf{d},V))$ be the twisted Galois action defined by $T(\sigma)=\varphi^{-1}\circ (id\times S(\sigma))\circ\varphi$, where $\sigma$ is an arbitrary element of $\Gal(L|K)$ and
$S(\sigma):\Spec L\to \Spec L$ is induced by $\sigma^{-1}: L\to L$. 
We call $T$ the Galois action corresponding to $\varphi$. 
\end{defn}

Let the variety $X$ be a form of an admissible flag variety and assume that its automorphism group is not bounded. 
In the next lemma we will examine the effect of the commutation of the automorphism group of $X$ (viewed as a subgroup of the automorphism group of the corresponding flag variety) 
and the corresponding twisted Galois action.

\begin{lem}
\label{splitting}
Let $K$ be a field containing all roots of unity. Let the $K$-variety $X$ be a form of an admissible flag variety. Let $L$ be a splitting field for $X$ such that $L|K$ is a Galois extension.
Assume that the automorphism group $\Aut_K(X)$ is not bounded. Let $X\times \Spec L\cong\Fl(\mathbf{d},V)$ (where $V$ is an $L$-vector space), and let $T:\Gal(L|K)\to \Aut_K(\Fl(\mathbf{d},V))$ be the corresponding twisted Galois action.
We can choose a basis $b$ of $V$ such that it splits as $b=b_1\cup b_2$ ($b_1,b_2\neq\emptyset$), giving rise to a direct sum decomposition $V=V_1\oplus V_2$ 
such that $\forall\sigma\in\Gal(L|K)$: $T(\sigma)=a_\sigma\circ B_b(\sigma)$ (see Definition \ref{Galois action}),
where $a_\sigma\in \Aut_L(\Fl(\mathbf{d},V))=\PGL(V)$ respects this decomposition, i.e. an arbitrary lift $c_\sigma\in \GL(V)$ of $a_\sigma$ is contained in $\GL(V_1)\times \GL(V_2)<\GL(V)$.    
\end{lem}

\begin{proof}
The isomorphism $X\times\Spec L\cong\Fl(\mathbf{d},V)$ induces an isomorphism $\Aut_L(X\times\Spec L)\cong\Aut_L(\Fl(\mathbf{d},V))=\PGL(V)$. Let $n=\dim V$, and fix a finite order element $g\in\Aut_K(X)$ with order larger than $n!$. 
It exists by the previous theorem. $g$ can be viewed as an element in $\PGL(V)$ since $\Aut_K(X)\leqq \Aut_L(X\times\Spec L)$.\\
 Let $h$ be a fixed lift of $g$ to $\GL(V)$ such that the order of $h$ is equal to the order of $g$, it exists since $K$ contains all roots of unity. 
 Notice that $h$ is of finite order, hence it is semisimple (since we are in characteristic 0). Let $b$ be a basis of $V$ consisting of eigenvectors of $h$. Again, this basis exists as $K$ contains all roots of unity and $h$ is of finite order
 (hence its eigenvalues are roots of unity).\\
 Let $V\cong V_{\lambda_1}\oplus V_{\lambda_2}\oplus...\oplus V_{\lambda_r}$ be the direct sum decomposition corresponding to the eigenspaces of $h^{n!}$. Since $g^{n!}\neq 1$, 
the linear transformation $h^{n!}$ cannot be a scalar multiply of the identity, therefore it has at least two distinct eigenspaces, i.e. $r\geqq 2$.
Let $V_1=V_{\lambda_1}$ and $V_2=V_{\lambda_2}\oplus...\oplus V_{\lambda_r}$. The basis $b$ splits as $b_1\cup b_2$, where $b_i$ is a basis of $V_i$. Indeed, an eigenspace of $h^{n!}$ is a direct sum of the eigenspaces of $h$.\\
Moreover, since $h$ is chosen to be of finite order: $h\circ A_b(\sigma)=A_b(\sigma)\circ h$, as $K$ contains all roots of unity by assumption.\\
The action of $g$ on $X\times \Spec L$ commutes with the natural Galois action. Indeed the action of $g$ derives from a group action on $X$, while the natural Galois action derives form a group action on $\Spec L$. 
Using the isomorphism $\Aut_L(X\times \Spec L)\cong \PGL(V)$, this leads us to $g\circ (a_\sigma\circ B_b(\sigma))=(a_\sigma\circ B_b(\sigma))\circ g$. 
Since $h$ and $A_b(\sigma)$ commutes, the same holds for $g$ and $B_b(\sigma)$, hence $g\circ a_\sigma=a_\sigma\circ g\in  \PGL(V)$.\\ 
Lift this equation to $\GL(V)$:
$hc_\sigma=\nu_\sigma c_\sigma h$, where $c_\sigma$ is an arbitrary lift of $a_\sigma$, and $\nu_\sigma\in L$ only depends on $\sigma$ (as we keep $h$ fixed throughout our argument).\\
Let $v_1,v_2,...,v_n\in V$ be a basis consisting of eigenvectors of $h$,  i.e. $hv_i=\mu_i v_i$ ($\mu_i\in L$; $i=1,....,n$). Consider the basis $c_\sigma v_1,c_\sigma v_2,..., c_\sigma v_n$, it is also a basis consisting of eigenvectors of $h$.
Indeed, $h(c_\sigma v_i)=\nu_\sigma c_\sigma hv_i=\nu_\sigma\mu_i (c_\sigma v_i)$. Since the eigenvalues of $h$ are uniquely determined, multiplication with $\nu_\sigma$ must permute them. Therefore $\nu_\sigma$ is a root of unity, with order less than or equal to $n$ ($\forall \sigma\in\Gal(L|K)$). Hence $h^{n!}c_\sigma=c_\sigma h^{n!}$. Therefore $c_\sigma\in \GL(V_1)\times \GL(V_2)<\GL(V)$ ($\forall \sigma\in\Gal(L|K)$). 
\end{proof}

\begin{rem}
A similar, but much more technical, statement can be formulated including the case of non-admissible flag varieties. Since we will not use it, we decided only to state the simpler version which applies to admissible flags.\\
\end{rem}

\section{Proof of the Main Theorem}
\label{Proof}
The strategy for the proof is the following. Instead of working with $X$, we will consider a flag variety equipped with a twisted Galois action. 
Using the splitting established in Lemma \ref{splitting} and the constructions introduced in Section \ref{rational maps}, we will build a Galois equivariant morphism from the flag variety to a lower dimensional variety, 
which is isomorphic to the Galois equivariant projection morphism of a vector bundle. Finally, by the use of Galois descent, we achieve the desired result.\\

If $\Aut_K(X)$ is bounded, then the claim of the main theorem (Theorem \ref{MainThm}) holds, so in the followings we assume otherwise.
\subsection{Setup of the proof of Theorem \ref{MainThm}}
\subsubsection{Notation}
Let $K$ be a field of characteristic $0$, containing all roots of unity. Let the  $K$-variety $X$ be a form of an admissible flag variety.\\ 
We will use the notations of $\phi_1, \phi_2$, $\psi$ and $U_1$, $U_2$, $U$  introduced in Section \ref{rational maps} (see also Remark \ref{L-vb}).\\
Let $L$ be a splitting field for $X$ such that $L|K$ is a finite Galois extension (see Definition \ref{SF}). 
Let $X\times \Spec L\cong \Fl(\mathbf{d_0},V)$  (where $V$ is an $L$-vector space), and let $T:\Gal(L|K)\to\Aut_K(\Fl(\mathbf{d_0},V))$ be the corresponding Galois action
(see Definition \ref{NatGalAct}).
Let $b$ be the basis of $V$ established in Lemma \ref{splitting}, $b=b_1\cup b_2$ and $V=V_1\oplus V_2$ be the corresponding decompositions.
By enlarging $L$ if necessary, we can assume that  $\phi_1, \phi_2$, $\psi$  and $U_1$, $U_2$, $U$ are defined over $L$ (see Remark \ref{L-vb}), Proposition \ref{vb} holds over $L$, 
moreover we can require that the sets of $L$-rational points in the complements of $U_1$, $U_2$ and $U$ are dense (see Remark \ref{L-vb}).
Let 
\begin{gather*}
A=A_b:\Gal(L|K)\to\Lin_K(V),\\  
B=B_b:\Gal(L|K)\to\Aut_K(\Fl(\mathbf{d_0},V)) 
\end{gather*}
be the corresponding twisted Galois actions (see Definition \ref{Galois action}). 
Finally, let $n=\dim V$ and $n_i=\dim V_i$ ($i=1,2$).\\
There are three different cases depending on the sequence $\mathbf{d_0}=(d_{0,1}<d_{0,2}<...<d_{0,r})$ and on $\dim V_1=n_1$.  Case 1: $d_{0,r}\leqq n_1$, Case 2: $n_1<d_{0,1}$ and Case 3: $d_{0,1}\leqq n_1<d_{0,r}$.
All of them should be handled similarly. \\

In the followings we will explicitly deal with Case 3. This contains all the necessary techniques and calculations involved in Case 1 and Case 2. At the end of each step we remark some of the necessary changes to deal with the other cases.

\subsubsection{Construction of the Galois actions on the target spaces}
\label{actions}
Let's assume Case 3. We will investigate $\psi$, at the end of the subsection we will note the changes for the other two cases.
Split $\mathbf{d_0}$ as $\mathbf{d}=(d_1<...<d_p)$ and $\mathbf{e}=(e_1<...<e_q)$, where $d_p\leqq n_1<e_1$ and 
$\mathbf{d_0}=(d_1<...<d_p<e_1...<e_q)$. The basis $b_1$ and $b_2$ induce the  following actions. 
\begin{gather*}
A_1=A_{b_1}:\Gal(L|K)\to\Lin_K(V_1)\\ 
B_1=B_{b_1}:\Gal(L|K)\to\Aut_K(\Fl(\mathbf{d},V_1))\\ 
A_2=A_{b_2}:\Gal(L|K)\to\Lin_K(V_2)\\
B_2=B_{b_2}:\Gal(L|K)\to\Aut_K(\Fl(\mathbf{e}-n_1,V_2))
\end{gather*}
By Lemma \ref{splitting} $\forall \sigma\in\Gal(L|K)$: $T(\sigma)=a_\sigma\circ B_b(\sigma)$ ($a_\sigma\in\PGL(V)$), 
and an arbitrary lift of $a_\sigma$, denoted by $c_\sigma$, splits, i.e. $c_\sigma\in \GL(V_1)\times\GL(V_2)$. For every $\sigma\in \Gal(L|K)$ fix a lift $c_\sigma$, 
and let $c_{\sigma,1}\in\GL(V_1)$ and $c_{\sigma,2}\in\GL(V_2)$ be its components. Let $a_{\sigma,1}\in\PGL(V_1)$ and $a_{\sigma,2}\in\PGL(V_2)$ be the images of $c_{\sigma,1}$
and $c_{\sigma,2}$ respectively.
Since all steps in our construction was compatible with the decomposition $V=V_1\oplus V_2$, 
\begin{gather*}
Q_1:\Gal(L|K)\to \Aut_K(\Fl(\mathbf{d}, V_1))\\
\sigma\mapsto a_{\sigma,1}\circ B_1(\sigma),\\
Q_2:\Gal(L|K)\to \Aut_K(\Fl(\mathbf{e}-n_1, V_2))\\
\sigma\mapsto a_{\sigma,2}\circ B_2(\sigma)
\end{gather*}
define twisted Galois actions for $\Fl(\mathbf{d}, V_1)$ and $\Fl(\mathbf{e}-n_1, V_2)$ respectively. Putting them together 
\begin{gather*}
Q:\Gal(L|K)\to \Aut_K(\Fl(\mathbf{d},V_1)\times \Fl(\mathbf{e}-n_1,V_2))\\
\sigma\mapsto (a_{\sigma,1}\circ B_1(\sigma))\times (a_{\sigma,2}\circ B_2(\sigma))
\end{gather*}
defines a twisted Galois action on $\Fl(\mathbf{d},V_1)\times \Fl(\mathbf{e}-n_1,V_2)$.\\
In Case 1 and Case 2 we do not need to introduce the notations $\mathbf{d}$ and $\mathbf{e}$. In Case 1 we need to consider a $Q_1$-like action on $\Fl(\mathbf{d_0}, V_1)$ (we denote it by $R_1$), 
while in Case 2 we need to consider a $Q_2$-like action on $\Fl(\mathbf{d_0}-n_1, V_2)$ (we denote it by $R_2$).

\subsection{Steps of the proof of Theorem \ref{MainThm}}
\subsubsection{Galois equivariance of the rational maps $\phi_1, \phi_2$ and $\psi$ }
To show the Galois equivariance of $\phi_1, \phi_2$ and $\psi$  we need to check two things, the invariance of the open subvariety where the rational maps are defined ($U_1$, $U_2$ and $U$, respectively),
 and the equivariance of the corresponding morphisms from the open subvarieties to the target spaces.
\begin{lem}
The open subvarieties $U_1$, $U_2$ and $U$ are invariant under the Galois actions, i.e. $\forall\sigma\in\Gal(L|K)$ $T(\sigma)U_i=U_i$ $(i=1,2)$ and  $T(\sigma)U=U$.
\end{lem}
\begin{proof}
We consider the case of $U$, the proof for the other two cases are almost verbatim.\\
Let $\sigma\in\Gal(L|K)$ be an arbitrary element of the Galois group. First notice that $L$-rational points of a flag variety can be identified with the flags of the underlying vector space.\\ 
We will show that an $L$-rational point (i.e. a flag) belongs to $U$ if and only if it belongs to $T(\sigma)U$.
Notice that the $L$-rational points of the open subvariety $U\subset \Fl(\mathbf{d}<\mathbf{e},V)$  are  given by
\[\{Z_1<...<Z_p<W_1<...<W_q\vert Z_p\cap V_2=\{0\}, W_1+ V_2=V\}.\] 
$V_1$, $V_2$ and $V$ are invariant under the natural actions of $c_\sigma$ and $A(\sigma)$ by construction.
Hence 
\begin{multline*}
T(\sigma)(Z_1<...<Z_p<W_1<...<W_q)=\\
(c_\sigma\circ A(\sigma))(Z_1)<...<(c_\sigma\circ A(\sigma))(Z_p)<\\
(c_\sigma\circ A(\sigma))(W_1)<...<(c_\sigma\circ A(\sigma))(W_q)
\end{multline*}
satisfies the defining equation of the $L$-rational points of $U$ if and only if $Z_1<...<Z_p<W_1<...<W_q\in U$.\\
Consider the complements of $U$ and $T(\sigma)U$ as topological subspaces in the underlying topological space of  $\Fl(\mathbf{d}<\mathbf{e},V)$. They are homeomorphic Zariski closed sets, 
moreover they contain exactly the same set of $L$-rational points. We have chosen the field $L$ in such a way that the $L$-rational points in the complement of $U$ forms a dense set. 
Putting these together implies that complement of $U$ and $T(\sigma)U$ are equal. Hence $T(\sigma)U=U$ as open subvarieties.
\end{proof}

Recall the definitions of the Galois actions $R_1$, $R_2$ and $Q$ from Section \ref{actions}.
\begin{lem}
\leavevmode
\begin{enumerate}
\item The morphism $\phi_1: U_1\to\Fl(\mathbf{d_0},V_1)$ is equivariant for the twisted Galois actions $T$ and $R_1$.  
\item The morphism $\phi_2: U_2\to\Fl(\mathbf{d_0}-n_1,V_2)$ is equivariant for the twisted Galois actions $T$ and $R_2$. 
\item The morphism $\psi: U\to\Fl(\mathbf{d},V_1)\times \Fl(\mathbf{e}-n_1,V_2)$ is equivariant for the twisted Galois actions $T$ and $Q$. 
\end{enumerate}

\end{lem}
\begin{proof}
First notice that $L$-flag varieties can be covered by affine spaces $\mathbb{A}^m_L$, where $m$ is the appropriate dimension. Therefore  the $L$-rational points form a dense set (as the same holds for $\mathbb{A}^m_L$). 
Hence, to show that two morphisms whose domains and target spaces are built up from open subvarieties of $L$-flag varieties are equal, it is enough to show that they are equal on $L$-rational points, which can be identified with flags. 
(Also note that checking Galois equivariance is equivalent to checking equality of morphisms.)\\
From now on we will deal with the case of $\psi$ and note the necessary changes at the end of the proof for the other two cases.
The twisted Galois action $T$ on the $L$-rational points (i.e. on the flags) is given by the formula
\begin{multline*}
T(\sigma)(Z_1<...<Z_p<W_1<...<W_q)=\\
(c_\sigma\circ A(\sigma))(Z_1)<...<(c_\sigma\circ A(\sigma))(Z_p)<\\
(c_\sigma\circ A(\sigma))(W_1)<...<(c_\sigma\circ A(\sigma))(W_q)
\end{multline*}
where $Z_1<...<Z_p<W_1<...<W_q$ is an arbitrary flag of the open subvariety $U$ and $\sigma\in\Gal(L|K)$ is an arbitrary element of the Galois group. While the twisted Galois action $Q$ is given by the formula
\begin{multline*}
Q(\sigma)(S_1<...<S_p,T_1<...<T_q)=\\
((c_{\sigma,1}\circ A_1(\sigma))(S_1)<...<(c_{\sigma,1}\circ A_1(\sigma))(S_p),\\
(c_{\sigma,2}\circ A_2(\sigma))(T_1)<...<(c_{\sigma,2}\circ A_2(\sigma))(T_q))
\end{multline*}
where $(S_1<...<S_p,T_1<...<T_q)$ is an arbitrary $L$-rational point of the product variety $\Fl(\mathbf{d},V_1)\times \Fl(\mathbf{e}-n_1,V_2)$ and $\sigma\in\Gal(L|K)$ is an arbitrary element of the Galois group. 
Comparing these equations with the definition of $\psi$  shows that for verifying the Galois equivariance of $\psi$ it is enough to check that the followings hold.
\begin{gather*}
pr\circ(c_\sigma\circ A(\sigma))(Z)=(c_{\sigma,1}\circ A_1(\sigma))\circ pr(Z)\\
(c_\sigma\circ A(\sigma))(W)\cap V_2=(c_{\sigma,2}\circ A_2(\sigma))(W\cap V_2),
\end{gather*} 
where $Z$ and $W$ are arbitrary linear subspaces of $V$ and $pr: V \to V_1$ is the projection along $V_2$.\\
For the first equation, consider an arbitrary vector $v\in V$. It can be written as $v=v_1+v_2$ where $v_i\in V_i$ ($i=1,2$).
\begin{gather*}
(pr\circ c_\sigma\circ A(\sigma))(v)=(c_{\sigma,1}\circ A_1(\sigma))(v_1)=(c_{\sigma,1}\circ A_1(\sigma)\circ pr)(v)
\end{gather*}
Hence the first equation is satisfied. For the second one, let $W\leqq V$ be an arbitrary linear subspace.
\begin{multline*}
(c_\sigma\circ A(\sigma))(W)\cap V_2=(c_\sigma\circ A(\sigma))(W)\cap (c_\sigma\circ A(\sigma))(V_2)=\\
(c_\sigma\circ A(\sigma))(W\cap V_2)=(c_{\sigma,2}\circ A_2(\sigma))(W\cap V_2),
\end{multline*}
where we used that $V_2$ is invariant under $c_\sigma\circ A(\sigma)$ and that $c_\sigma\circ A(\sigma)$ is a bijection from $V$ to $V$. Hence the second equation is satisfied too, which shows that $\psi$ is Galois equivariant.\\
For the case of $\phi_1$ we need to perform the steps corresponding to the $Z_1<...<Z_p$-part of the above argument,
meanwhile for the case of $\phi_2$ we need to perform the steps corresponding to the $W_1<...<W_q$-part.
\end{proof}

\subsubsection {Galois equivariance of the vector bundle structure}
\begin{lem}
\label{ConstrG-vb}
The vector bundles $(U_1,\Fl(\mathbf{d_0},V_1),\phi_1)$, $(U_2,\Fl(\mathbf{d_0}-n_1,V_2),\phi_2)$ and $(U,\Fl(\mathbf{d},V_1)\times\Fl(\mathbf{e}-n_1,V_2),\psi)$ are Galois equivariant.
In other words, the Galois actions respect the addition and twist (by the corresponding element of the Galois group) the multiplication by scalar operations.
\end{lem}
\begin{proof}
Again using the fact that $L$-rational points of $L$-flag varieties form a dense set, it is enough to check that the vector bundle structure is respected on the $L$-rational points, i.e. on the flags.
As usual we assume the case of $(U,\Fl(\mathbf{d},V_1)\times\Fl(\mathbf{e}-n_1,V_2),\psi)$  and note the necessary changes for the other two cases at the end of the proof.\\
Let $Z_1<...<Z_p<W_1<...<W_q\in U$ be an arbitrary flag lying over $(S_1<...<S_p, T_1<...<T_q)\in\Fl(\mathbf{d}, V_1)\times\Fl(\mathbf{e}-n_1, V_2)$. As we have seen before, 
its image under $T(\sigma)$ ($\sigma\in\Gal(L|K)$) is the flag 
\begin{multline*}
(c_\sigma\circ A(\sigma))(Z_1)<...<(c_\sigma\circ A(\sigma))(Z_p)<\\
(c_\sigma\circ A(\sigma))(W_1)<...<(c_\sigma\circ A(\sigma))(W_q)\in U 
\end{multline*}
which lies over 
\begin{multline*}
((c_{\sigma,1}\circ A_1(\sigma))(S_1)<...<(c_{\sigma,1}\circ A_1(\sigma))(S_p),\\
(c_{\sigma,2}\circ A_2(\sigma))(T_1)<...<(c_{\sigma,2}\circ A_2(\sigma))(T_q))\in\Fl(\mathbf{d}, V_1)\times\Fl(\mathbf{e}-n_1, V_2).
\end{multline*}
Using these formulas and the equations \eqref{S}, \eqref{W'} and \eqref{W} which construct the flag $Z_1<...<Z_p<W_1<...<W_q\in U$ from $(f,g)\in E<\Hom(S_p,V_2)\times \Hom(V_1,V_2/T_1)$ 
(and the corresponding equations for the image of the flag),
we can see that the image of the flag corresponds to 
\begin{gather*}
(c_{\sigma,2}\circ A_2(\sigma))\circ f \circ ((c_{\sigma,1}\circ A_1(\sigma))^{-1}\in\Hom((c_{\sigma,1}\circ A_1(\sigma))(S_p),V_2) \\
\overline{(c_{\sigma,2}\circ A_2(\sigma))}\circ g \circ ((c_{\sigma,1}\circ A_1(\sigma))^{-1}\in\Hom(V_1,V_2/(c_{\sigma,2}\circ A_2(\sigma))(T_1)),
\end{gather*}
where 
\[\overline{(c_{\sigma,2}\circ A_2(\sigma))}:V_2/T_1\to V_2/(c_{\sigma,2}\circ A_2(\sigma))(T_1)\] 
is the $\sigma$-linear homomorphism induced by $c_{\sigma,2}\circ A_2(\sigma): V_2 \to V_2$. Therefore we have a Galois action on the vector bundle structure which respects the addition and twists the multiplication by scalar operations. 
(Observe that the formula of the action does not depend on the choice of the lift $c_\sigma$ as both $c_{\sigma,1}$ and $c_{\sigma,2}$ derives from the same lift.)\\ 
Again for the other two cases we only need to carry out half of the proof. For the case of $(U_1,\Fl(\mathbf{d_0},V_1),\phi_1)$ we need the part which corresponds to $Z_1<...<Z_p$ and $f$, 
for the case of $(U_2,\Fl(\mathbf{d_0}-n_1,V_2),\phi_2)$ we need the other half which corresponds to $W_1<...<W_q$ and $g$.
\end{proof}

\subsubsection {Galois descent}
We can finish our proof by the help of Galois descent. 

\begin{proof}[Proof of Theorem \ref{MainThm}]
For Case 3 we can summarize the results of the previous lemmas in the following way. We constructed a Galois equivariant commutative diagram of $L$-varieties, where $\psi$ is a Galois equivariant projection of a vector bundle structure.
\[
\xymatrix{
U \ar[d]^\psi \ar@{^{(}->}[r] & \Fl(\mathbf{d}<\mathbf{e},V) \ar@{-->}[dl]\\
\Fl(\mathbf{d},V_1)\times\Fl(\mathbf{e}-n_1,V_2)
}
\]
Generally (including all three cases) we can say that, there is a Galois equivariant commutative diagram of $L$-varieties
\[
\xymatrix{
W \ar[d]^\pi \ar@{^{(}->}[r]^i & \Fl(\mathbf{d_0},V) \ar@{-->}[dl]\\
Y
}
\]
where $W$ is an open subvariety of $\Fl(\mathbf{d_0},V)$ and $(W,Y,\pi)$ is a Galois equivariant vector bundle (Lemma \ref{ConstrG-vb}).\\
Finally, we can use results of Galois descent (Theorem \ref{equivalence} and Theorem \ref{G-vb}) to achieve a commutative diagram over $K$ with the same properties. 
\[
\xymatrix{
W^* \ar[d]^{\pi^*} \ar@{^{(}->}[r] & X \ar@{-->}[dl]\\
Y^*
}
\]
$W^*$ is an open subvariety of $X$, therefore they are birational. Since $W^*$ is a vector bundle over $Y^*$,  $W^*$ is birational to $\mathbb{P}^m\times Y^*$ for some $m>0$. 
Putting these together shows that $X$ is birational to $\mathbb{P}^m\times Y^*$.
\end{proof}

\begin{rem}
\label{NANGE}
Now we can reflect on the role of the admissibility condition.
In the proof above we showed that we can endow $\Fl(\mathbf{d},V_1)$,  $\Fl(\mathbf{e}-n_1,V_2)$ and $\Fl(\mathbf{d},V_1)\times \Fl(\mathbf{e}-n_1,V_2)$ with twisted Galois actions
which makes the morphism $\phi_1$, $\phi_2$ and $\psi$ Galois equivariant.
If $\Fl(\mathbf{d},V)$ is non-admissible then $\Fl(\mathbf{d},V_1)$ and  $\Fl(\mathbf{e}-n_1,V_2)$ must be admissible. 
In this case there exists a pair $(T,\sigma)$ where $T:\Gal(L|K)\to\Aut_K(\Fl(\mathbf{d},V))$ is a twisted Galois actions and $\sigma\in\Gal(L|K)$  is such that 
$T(\sigma)=a_\sigma\circ B_b(\sigma)$ is `dimension-swapping', i.e $a_\sigma\notin\PGL(V)$ (see Remark \ref{NaturalAction2}). 
Because of this dimension-swap we cannot construct Galois actions on the target spaces which makes the morphisms $\phi_1$, $\phi_2$ and $\psi$ Galois equivariant.   
\end{rem}

\section{Brauer-Severi surfaces}
\label{BS2}
We  analyze the question of non-ruledness of non-trivial Brauer-Severi surfaces. We need to introduce a couple of new concepts. We will not explain them in full detail, the interested reader is referred to 
\cite{GS06}, \cite{Ja} and \cite{Ko16} for further information on the subject.
During this section we can relax the condition that the ground field is of characteristic zero. 

\begin{defn}
Let $K$ be a field, and let $X$ and $Y$ be Brauer-Severi varieties over $K$ with dimensions $n$ and $m$ respectively. Let $\varphi: X\dashrightarrow Y$ be a rational map. $\varphi$ is called twisted linear if  it is linear over $\overline{K}$, 
i.e. the composite map $\mathbb{P}^n \cong X\times\Spec \overline{K}\dashrightarrow Y\times\Spec \overline{K}\cong \mathbb{P}^m$ is linear.\\
We call the pair $(X,\varphi)$ a twisted linear subvariety of $Y$ if the composite map is induced by a linear injection. (Notice that in the case of a twisted linear subvariety $\varphi$ can be extended to a morphism of varieties.) 
If no confusion can arise we denote the pair $(X,\varphi)$ simply by $X$.\\
We call $X$ a minimal twisted linear subvariety of $Y$, if it is a twisted linear subvariety and has minimal dimension amongst the twisted linear subvarieties. 
By Theorem 28 in \cite{Ko16} the isomorphism class of a minimal twisted linear subvariety is well defined.\\
We call the Brauer-Severi variety $Y$ minimal if the only twisted linear subvariety of $Y$ is itself (up to isomorphism).\\
For an arbitrary Brauer-Severi variety $P$ we will denote a fixed minimal twisted linear subvariety by $P^{\m}$. 
\end{defn}

\begin{lem}
\label{LDNT}
Let $K$ be a field  and let $X$ be a Brauer-Severi curve or a Brauer-Severi surface over $K$. $X$ is non-trivial if and only if $X$ is minimal.
\end{lem}
\begin{proof}
If $X$ is a non-minimal Brauer-Severi curve then $X$ has a $0$-dimensional twisted linear subvariety, i.e $X$ has a $K$-rational point. Then by Ch\^atelet's theorem $X$ is trivial (Theorem 5.1.3 in \cite{GS06}).\\
If $X$ is a non-minimal Brauer-Severi surface then $X$ either has a $K$-rational point or a one codimensional  twisted linear subvariety. In both cases $X$ is trivial by versions of Ch\^atelet's theorem.\\
The other directions are trivial.
\end{proof}

\begin{defn}[Definition-Lemma 31 in \cite{Ko16}]
We call two Brauer-Severi varieties $X$ and $Y$ similar or Brauer equivalent if  $X^{\m}\cong Y^{\m}$.
\end{defn}

\begin{rem}
There is a canonical correspondence between central simple algebras and Brauer-Severi varieties over a given field $K$ (Theorem 5.1 in\cite{Ja}). We can also introduce the Brauer equivalence relation on the central simple algebras in a natural way.
The canonical correspondence between central simple algebras and Brauer-Severi varieties respects these equivalence relations.\\
Furthermore we can endow the central simple algebras with operations (tensor product and taking the opposite algebra), which respects the Brauer equivalence relation and turn the equivalence classes into a commutative group, 
called the Brauer group (Chapter 2.4 in \cite{GS06}).\\
A similar construction can be carried out purely geometrically.
\end{rem}

\begin{thm}
\label{GBG}
We can introduce operations on Brauer-Severi varieties which turns the Brauer equivalence classes into a commutative group which is naturally isomorphic to the Brauer group. 
(For further details see Section 4 and 5 of \cite{Ko16}.)
\end{thm}

\begin{rem}
If $X$ and $Y$ are Brauer-Severi varieties we will use the notation $X\otimes Y$ for the binary operation introduced in Theorem \ref{GBG}. 
We will use the notation $X^{\otimes m}$ to denote the $m$-fold `product' of $X$ with itself ($m\in\mathbb{Z}^+$).
\end{rem}

\begin{thm}[Amitsur's theorem, Proposition 45 in \cite{Ko16}]
\label{Amitsur}
Let $X$ and $Q$ be Brauer-Severi varieties. The following two conditions are equivalent: $Q$ is similar to $X^{\otimes m}$ for some positive integer $m$;
there is a rational map $\varphi: X\dashrightarrow Q$.
\end{thm}

\begin{defn}
Let $K$ be a field and $X$ be a projective $K$-variety. The index of $X$ is the greatest common divisor of the degrees of all $0$-cycles on $X$. It is denoted by $\ind (X)$.
\end{defn}

\begin{lem}[Lemma 51 in \cite{Ko16}]
\label{id}
Let $X$ be a Brauer-Severi variety and $m$ be a positive integer, then the index of $X^{\otimes m}$ divides the index of $X$.
\end{lem}

\begin{thm}[Theorem 53 in \cite{Ko16}]
\label{indiv}
Let $X$ be a Brauer-Severi variety. Then $\ind(X)=\ind(X^{\m})=\dim X^{\m}+1.$
\end{thm}

\begin{lem}
Let $K$ be a field and $X$ be a Brauer-Severi surface over $K$. If $X$ is ruled then either $X$ is trivial or there exists a rational map $\varphi:X\dashrightarrow Q$, where $Q$ is a non-trivial Brauer-Severi curve.
\end{lem}
\begin{proof}
If $X$ is ruled then it is birational to $\mathbb{P}^1_K\times Q$, where $Q$ is a smooth projective curve. Notice that if $Q$ is birational to $\mathbb{P}^1_K$, then $X$ has $K$-rational points, therefore $X$ is trivial by Ch\^atelet's theorem.
So we can assume that $Q$ is not isomorphic to the projective line.\\
Denote $Q\times\Spec{\overline{K}}$ by $Q_{\overline{K}}$. Since $X$ is a Brauer-Severi surface, $\mathbb{P}^1_{\overline{K}}\times Q_{\overline{K}}$ is rational.\\
Therefore we can take a general rational curve $c:\mathbb{P}^1_{\overline{K}}\to\mathbb{P}^1_{\overline{K}}\times Q_{\overline{K}}$ 
(i.e. we can take a general morphism of the projective line to $\mathbb{P}^1_{\overline{K}}\times Q_{\overline{K}}$) and 
we can compose $c$ with the canonical projection $\mathbb{P}^1_{\overline{K}}\to\mathbb{P}^1_{\overline{K}}\times Q_{\overline{K}}\to Q_{\overline{K}}$. Since $c$ is general, the composite is dominant
(i.e. the rational curve does not lie in a fiber over $Q_{\overline{K}}$). 
Hence we get a non-trivial morphism from a projective line to the smooth curve $Q_{\overline{K}}$. This implies that $Q_{\overline{K}}$ is isomorphic to the projective line, i.e. $Q$ is a Brauer-Severi curve. 
As $Q$ is not isomorphic to the projective line, $Q$ is non-trivial.\\
The composite $X\dashrightarrow\mathbb{P}^1_K\times Q\to Q$, where the first map is the birational isomorphism giving the ruledness and the second is the canonical projection, gives a rational map $X\dashrightarrow Q$ 
from $X$ to a non-trivial Brauer-Severi curve.
\end{proof}

\begin{proof}[Proof of Theorem \ref{BS2T}]
If the Brauer-Severi surface $X$ is trivial, then it is ruled.\\
Assume that $X$ is non-trivial and ruled. By the previous lemma there is a rational map $\varphi: X\dashrightarrow Q$, where $Q$ is a non-trivial Brauer-Severi curve.\\
By Amitsur's theorem (Theorem \ref{Amitsur}) $Q$ is similar to $X^{\otimes m}$ for some positive integer $m$. Hence $Q^{\min}\cong (X^{\otimes m})^{\m}$ by the definition of similarity.\\
We can consider indices: 
$$2=\dim Q+1=\dim Q^{\m}+1=\ind(Q^{\m})=\ind((X^{\otimes m})^{\m})=\ind(X^{\otimes m}),$$ 
by Lemma \ref{LDNT} and by Theorem \ref{indiv}. 
On the other hand $\ind(X^{\otimes m})$ divides $\ind(X)$ by Lemma \ref{id}, and 
$$\ind(X)=\dim X^{\m}+1=\dim X+1=3,$$ 
by Theorem \ref{indiv} and Lemma \ref{LDNT}. Since $2$ does not divide $3$, we arrived to a contradiction.
Hence a non-trivial Brauer-Severi surface cannot be ruled. This finishes the proof.  
\end{proof}

\begin{rem}
Brauer-Severi surfaces corresponds canonically to degree three central simple algebras (Theorem 5.1 in\cite{Ja}). (The degree of a central simple algebra is the square root of its dimension, it is a positive integer.) 
By Wedderburn's theorem degree three central simple algebras are cyclic algebras (Chapter 15.6 in \cite{Pi82}). Moreover, if $K$ is a field of characteristic zero containing all roots of unity, 
then cyclic algebras over $K$ are given by the following presentation:
$$K<x_1,x_2| x_1^m=a, x_2^m=b, x_1x_2=\omega x_2x_1>,$$
where $m\in\mathbb{Z}^+$, $a,b\in K^*$ and $\omega$ is a primitive $m$-th root of unity (Corollary 2.5.5 in \cite{GS06}).\\
We call a central simple algebra over $K$ split  if it is isomorphic to a matrix ring over $K$. It is equivalent with the corresponding Brauer-Severi variety being trivial. 
A cyclic algebra of the above presentation (where $K$ is a field of characteristic zero containing all roots of unity) is split if and only if $b$ is a norm from the field extension $K(\sqrt[m]a)|K$ (Corollary 4.7.7 in \cite{GS06}).\\
Putting these together, one can show that 
$$\mathbb{C}(t_1,t_2)<x_1,x_2| x_1^3=t_1, x_2^3=t_2, x_1x_2=e^{2\pi i/3} x_2x_1>$$
corresponds to a non-trivial Brauer-Severi surface over a field of characteristic zero containing all roots of unity.   
\end{rem}

\end{document}